\documentclass[12pt,reqno]{amsart}
\usepackage{amsmath,amsfonts,amsthm,amssymb,amscd,stmaryrd,url,hyperref,subfig,cancel,todonotes,bbm}

\usepackage{longtable,tabularx,tabulary}
\usepackage{titletoc}
\usepackage{float,mathtools}
\usepackage{color}
\usepackage{anysize,placeins,enumitem}
\usepackage{latexsym}
\usepackage{graphicx}
\usepackage{epsfig}
\usepackage{booktabs}
\usepackage{comment}
\usepackage{xcolor}
\usepackage{textgreek}
\usepackage{textcomp}
\usepackage{cases}
\usepackage{url}

\usepackage{todonotes}
\usepackage{mathrsfs}
\usepackage{caption}
\usepackage{hyperref}

\newtheorem{theorem}{Theorem}[section]

 \newtheorem{proposition}[theorem]{Proposition}

\theoremstyle{remark}
\newtheorem*{remark}{Remark}

\numberwithin{equation}{section}
\reversemarginpar %forces the margin note to appear on the inside of the margin

\begin{document}

\title[Minimal zero-free regions]{Minimal zero-free regions for results on primes between consecutive perfect $k$th powers}
\author[E.~S.~Lee]{Ethan~Simpson~Lee}
\address{University of the West of England, School of Computing and Creative Technologies, Coldharbour Lane, Bristol, BS16 1QY} 
\email{ethan.lee@uwe.ac.uk}
\urladdr{\url{https://sites.google.com/view/ethansleemath/home}}

\maketitle

\begin{abstract}
We compute minimal zero-free regions for the Riemann zeta-function of the Littlewood form which ensure there is always a prime between consecutive perfect $k$th powers. Our computations cover powers $k\geq 65$ and quantify how far we are away from proving certain milestones toward an infamous open problem (Legendre's conjecture). In addition, we prove there is always a prime between consecutive perfect $86$th powers, and identify an integer sequence (that is a subset of the positive integers) for which there is always a prime between consecutive $70$th powers.
\end{abstract}

\section{Introduction}

Let $\zeta$ denote the Riemann zeta-function. Recall the Riemann Hypothesis (RH) postulates that every non-trivial zero $\varrho = \beta + i\gamma$ of $\zeta$ satisfies $\beta=1/2$ and $\gamma\neq 0$. The Riemann height is a constant $H_R > 0$ such that the RH is known to be true for all $|\gamma|\leq H_R$. Platt and Trudgian \cite{PlattTrudgianRH} have announced $H_R = 3\,000\,175\,332\,800$ is admissible.

A famous open problem from Legendre postulates there is at least one prime between consecutive perfect squares, i.e., in the interval $(n^{2},(n+1)^{2})$ for every integer $n\geq 1$. A proof of Legendre's conjecture is presently out of reach, even under the assumption of the RH. On the other hand, Ingham \cite{Ingham1937} proved Theorem \ref{thm:Ingham} by applying results on the distribution of the zeros of $\zeta$.

\begin{theorem}[Ingham]\label{thm:Ingham}
If $n$ is sufficiently large, then there is at least one prime in the interval $(n^{3}, (n+1)^{3})$.
\end{theorem}

In recent work, Cully-Hugill \cite{CH_23} has proved there is at least one prime in the interval $(n^{3}, (n+1)^{3})$ for every integer $n \geq \exp(\exp(32.892))$. This result makes progress toward a complete explicit version of Theorem \ref{thm:Ingham}, but the range of $n$ that remains to be made explicit is largely impervious to modern techniques. 

This naturally leads to the problem of computing the least $k \geq 3$ such that we can prove (unconditionally) that there is at least one prime in the interval $(n^{k}, (n+1)^{k})$ for every integer $n \geq 1$. Several authors \cite{CH_23, CH_J-I, CH_J_24, Dudek_16, MattnerThesis} have proved progressively stronger results; the latest result in this direction claims $k = 90$ is admissible (see \cite{CH_J_24}). Using similar techniques, we prove Theorem \ref{thm:CHJ_extended}, which improves these computations. 

\begin{theorem}\label{thm:CHJ_extended}
For every integer $n \geq 1$, there is at least one prime in the interval $(n^{86}, (n+1)^{86})$.
\end{theorem}

For any fixed $k$, the method to prove there is at least one prime in the interval $(n^{k}, (n+1)^{k})$ for every integer $n \geq 1$ requires three ingredients:
\begin{enumerate}
    \item Zero-free regions for $\zeta$ of the Littlewood form (or stronger).
    \item Upper bounds for the number $N(\sigma,T)$ of non-trivial zeros $\varrho = \beta + i\gamma$ of $\zeta$ in the rectangle $\sigma \leq \beta \leq 1$ and $0 < \gamma \leq T$.
    \item Explicit descriptions for the error in the truncated Riemann--von Mangoldt formula.
\end{enumerate}
In general, to prove that smaller $k$ are admissible, recent works have applied the latest knowledge on the distribution of the zeros of $\zeta$ and exercised tighter control over the error in the truncated Riemann--von Mangoldt formula. While this has been a fruitful endeavour, further improvements to the error in the truncated Perron formula are limited, as the present results are quite sharp save for minor refinements to the constants within. 

In this paper, we investigate two alternative strategies to prove complete results (i.e., results which hold on $n\geq 1$) for $k < 86$. First, we fix $k = 70$ and prove a complete result on a subset of the natural numbers (see Theorem \ref{thm:CHJ_extended_bigger_n}). 
Second, we compute minimal zero-free regions for $\zeta$ which are sufficient to prove complete results for $k < 86$ (see Theorem \ref{thm:night}). 

\begin{theorem}\label{thm:CHJ_extended_bigger_n}
Fix $N \geq 0$ and let $a_n =n(1 + \lfloor N/n \rfloor)$. 
If $N \geq \exp\left(\frac{15\,951}{70}\right) - \exp\left(\frac{2\,135}{70}\right)$, then there is at least one prime in the interval $(a_{n}^{70}, a_{n+1}^{70})$ for every integer $n\geq 1$.
\end{theorem}

\begin{theorem}\label{thm:night}
Suppose that $k\in\{85,80,75,70\}$ is fixed and $z_k$, $T_k$ are the values associated to each choice of $k$ in Table \ref{tab:constants}. If $\zeta(\sigma + it) \neq 0$ in the region
\begin{equation*}
    1 - \frac{\log\log{|t|}}{z_k \log{|t|}} \leq \sigma \leq 1 
    \quad\text{and}\quad
    H_R < |t| \leq T_k , %< e^{271.24349} ,
\end{equation*}
then there is at least one prime in the interval $(n^k,(n+1)^k)$ for every integer $n\geq 1$. Computations for every $65\leq k\leq 85$ not covered by this theorem statement can be found in the main body of the paper (see Table \ref{tab:big2}).
\end{theorem}

\begin{table}[h!]
    \centering
    \begin{tabular}{|c|cc|}
        \hline
        $k$ & $z_k$ & $\log{T_k}$ \\
        \hline
        85 & 17.270 & 75.853 \\
        80 & 16.168 & 80.643 \\
        75 & 15.122 & 90.945 \\
        70 & 14.055 & 264.334 \\
        %65 & 12.987 & 739\,410.809 \\
        \hline
    \end{tabular}
    \caption{Table of admissible constants in Theorem \ref{thm:night}.}
    \label{tab:constants}
\end{table}

%\begin{example}
For example, Theorem \ref{thm:night} tells us that if $\zeta(\sigma + it) \neq 0$ in the thin rectangular region 
\begin{equation*}
    0.99168 \leq \sigma \leq 1 
    \quad\text{and}\quad 
    H_R < |t| \leq e^{264.334} %= 6.292\ldots\cdot 10^{114}
\end{equation*}
of the critical strip, then there is at least one prime in the interval $(n^{70},(n+1)^{70})$ for every integer $n\geq 1$. %To see this, note that $0.99168 < 1 - \log\log{|t|} / 14.055 \log{|t|}$ for every $H_R < |t| \leq \exp(271.244)$. 
This is interesting, because a proof of the result with $k = 70$ would be a significant theoretical milestone, yet the associated computation appears reasonable. Using the results in this paper, we can repeat this analysis for any $k\in [65,85]$, and the difficulty of these computations indicate how far away we are from complete proofs of important milestones toward Legendre's conjecture.
%\end{example}

%\begin{remark}
To prove Theorems \ref{thm:CHJ_extended_bigger_n}-\ref{thm:night}, we enhanced the usual methods by applying more than one bound for $N(\sigma,T)$. This strategy yields stronger computations for $67 \leq k \leq 85$, but offered no benefits outside this range. For example, without this refinement to the method, we could only prove Theorem \ref{thm:CHJ_extended_bigger_n} with $$N \geq \exp\left(\frac{580\,009}{70}\right) - \exp\left(\frac{2\,135}{70}\right) = 3.132\ldots\cdot 10^{3598}.$$
In the future, if stronger explicit bounds for $N(\sigma,T)$ with $H_R < T \leq \exp(80)$ of the right shape are proved, then this strategy can be extended to improve the computations in Theorems \ref{thm:CHJ_extended_bigger_n}-\ref{thm:night}. 
%\end{remark}

\begin{remark}
If $n > N$ in Theorem \ref{thm:CHJ_extended_bigger_n}, then $a_n = n$. Therefore, Theorem \ref{thm:CHJ_extended_bigger_n} also demonstrates there is at least one prime in the interval $(n^{70},(n+1)^{70})$ for every integer $$n > \exp\left(\frac{15\,951}{70}\right) - \exp\left(\frac{2\,135}{70}\right) = 9.189\ldots\cdot 10^{98}.$$
\end{remark}

%\begin{remark}
%For a fixed $65\leq k\leq 85$, once a minimal zero-free region has been identified using Theorem \ref{thm:night}, it suffices to verify there are no zeros of $\zeta$ in this region. It appears feasible that in future work, one can update the computational methods which have enabled us to verify the Riemann Hypothesis up to large heights, to verify the necessary zero-free region holds.  
%\end{remark}

\subsection*{Structure}

The remainder of this paper is dedicated to proving Theorems \ref{thm:CHJ_extended}-\ref{thm:night}. An important intermediate result, namely Proposition \ref{prop:theta_bounds}, establishes user-friendly parametrised bounds for a difference of prime-counting functions, which are formulated so they can be readily updated or adapted for use in future work. Since obtaining explicit estimates over primes in short intervals is often a technically delicate task, the parametrised bounds we prove may be of independent interest beyond the present paper. 

%The remainder of this paper is structured as follows. 
In Section \ref{sec:RZF}, we introduce several properties of the Riemann zeta-function $\zeta$. In Section \ref{sec:auxiliary_bounds}, we prove Proposition \ref{prop:theta_bounds} and apply these bounds to prove the auxiliary results we require. In Section \ref{sec:pbcp}, we combine these results to prove Theorems \ref{thm:CHJ_extended}-\ref{thm:night}. To avoid breaking the flow of the narrative, we have collected all referenced tables in Appendix \ref{app:tables}.

\subsection*{Computations}

To prove our results, we used \texttt{Python}. To supplement this work, our code to obtain all the computations in this paper are given here:  \href{https://github.com/EthanSLee/PrimesBetweenConsecutivePowers}{\texttt{press this link}}. 

\subsection*{Acknowledgements}

The author is grateful to Andrew Yang, Chiara Bellotti, Dave Platt, Daniel Johnston and several other colleagues for valuable feedback and insightful discussions. 

\section{Riemann zeta-function}\label{sec:RZF}

Results pertaining to the distribution of the non-trivial zeros of $\zeta$, which we denote by $\varrho = \beta + i\gamma$, play a central role in this paper. Therefore, we introduce several results pertaining to the distribution of these zeros in this section. All of the notation that is introduced in this section will be used throughout the rest of this paper. 

\subsection{Zero-free regions}\label{ssec:zfr}

There are three commonly-seen forms of zero-free region:
\begin{equation}\label{zfr regions}
    \sigma \ge 1- \frac{1}{Z_1\log t},\quad 
    \sigma \ge 1- \frac{\log\log t}{Z_2 \log t},\quad\text{and}\quad 
    \sigma\ge 1- \frac{1}{Z_3(\log t)^{2/3}(\log\log t)^{1/3}} .
\end{equation}
For $t\ge 3$, Yang has proved in his PhD thesis \cite{YangThesis} that $\zeta(\sigma + it)\neq 0$ in the regions defined in \eqref{zfr regions}, with 
\begin{equation}\label{z zerofree const}
    Z_1 = 4.862,\quad 
    Z_2 = 19.62,\quad\text{and}\quad 
    Z_3 = 51.34.
\end{equation}
These computations improve earlier results from \cite{mossinghoff2024explicit}, \cite{yang2024explicit}, and \cite{bellotti2024explicit} respectively. In addition, Ford \cite[Thm.~3]{ford2002zero} has proved that $\zeta(\sigma + it) \neq 0$ in the region $\sigma \geq 1 - Z(t)^{-1}/\log{t}$ and $t \geq 5.45\cdot 10^8$ such that
\begin{equation*}%\label{eqn:JT_def}
\begin{split}
    J(T) &= \min\left\{\frac{\log{T}}{4} + 1.8521, \frac{\log{T}}{6} + \log\log{T} + \log{3}\right\}
    \quad\text{and}\\
    Z(T) &= \frac{J(T)+0.685+0.155\log\log{T}}{\left(0.04962-\frac{0.0196}{J(T) + 1.15}\right)\log{T}} .
\end{split}
\end{equation*}
Combined with the zero-free regions in \eqref{zfr regions}, we have $\zeta(\sigma + it) \neq 0$ in the region 
\begin{equation*}
    1 - \frac{\log\log{|t|}}{\widehat{Z}(t) \log{|t|}} \leq \sigma \leq 1 
    \quad\text{and}\quad
    |t| \geq 5.45\cdot 10^8 ,
\end{equation*}
in which
\begin{equation}\label{eqn:Z_hat_def}
    \widehat{Z}(T) =
    \min\left\{ Z_2, 
    \min\{Z_1, Z(T)\} \log\log{T}, 
    \frac{Z_3 (\log\log{T})^{4/3}}{(\log{T})^{1/3}}
    \right\},
\end{equation}
which is a non-increasing function.

\subsection{Zero-density estimates}\label{ssec:zdb}

For $\frac{1}{2}\leq \sigma<1$, let
\begin{align*}
    N(\sigma, T) &= \#\{\varrho:\sigma < \beta < 1, 0 < \gamma < T\}
    \quad\text{and}\quad N(T) = N(0, T) .
\end{align*}
First, there exist constants $a_1, a_2, a_3$ such that
\begin{equation}\label{classical_T}
    \left|N(T) - \frac{T}{2\pi}\log{\frac{T}{2\pi e}} - \frac{7}{8}\right|\leq R(T)
\end{equation}
where $R(T) = a_1\log{T} + a_2\log\log{T} + a_3$ for all $T \geq H_R$. In recent work, Bellotti and Wong \cite{BellottiWong2025} have computed that $(a_1, a_2, a_3) = (0.10076, 0.24460,  8.08344)$ are admissible; their computations improve earlier work of Hasanalizade, Shen, and Wong \cite{HasanalizadeShenWongRiemann}. It follows from \eqref{classical_T} that if $T\geq 1$, then
\begin{equation}\label{eqn:a_bound_of_convenience}
    N(T) \leq \frac{T\log{T}}{2\pi} .
\end{equation}
Next, we import an upper bound for $N(\sigma,T)$ from Bellotti \cite{Bellotti}, who has computed explicit constants $C$, $B$, $\alpha_0$, $T_0$, $T_1$ (presented in Table \ref{tab:Bellotti}) such that if $T_0 \leq T \leq T_1$ and $\alpha_0 < \sigma < 1$, then 
\begin{equation}\label{eqn:Bellotti}
    N(\sigma,T) \leq C T^{B(1-\sigma)} .
\end{equation}
Further, computations from \cite{K_L_N_2018} and \cite[Appendix~B]{JohnstonYang} prove that if $0.6 < \sigma < 1$ and $T\geq H_R$, then
\begin{equation}\label{eqn:zd_K}
    N(\sigma, T) 
    \leq 17.253 (\log{T})^{5-2\sigma} T^{\frac{8}{3}(1-\sigma)} + 5.663 (\log T)^2 .
\end{equation}

\section{Auxiliary Bounds}\label{sec:auxiliary_bounds}

Recall the Chebyshev functions, denoted and defined by
\[
\theta(x) = \sum_{p \leq x} \log p \quad \text{and} \quad 
\psi(x) = \sum_{p^r \leq x} \log p,
\]
where the first sum is over primes $p$ and the second is over prime powers. To prove the main results of this paper, we require explicit lower bounds for a prescribed difference of $\theta$ functions. %Each of the results we prove in this section is an important component in the proof of our main results.
To establish the bounds we require, we prove the following technical result (Proposition \ref{prop:theta_bounds}), which might also be of independent interest. 

\begin{proposition}\label{prop:theta_bounds}
Fix $k \geq 11$ and let $h = k x^{1 - \frac{1}{k}}$. Let $0.6 < \sigma_1 < 1$, $\alpha_1= 1 + 1.93378 \cdot 10^{-8}$, $\alpha_2 = 1 + 1.936 \cdot 10^{-8}$, and $2T = x^{\alpha}$ such that $T \geq H_R$ and $$\frac{1}{k} < \alpha < \min\left\{\frac{1}{10},1 - \sigma_1\right\}.$$ Suppose that $\zeta(\sigma + it) \neq 0$ in the region $1 - \nu(2T) \leq \sigma \leq 1$ and $|t| \leq 2T$, where $$\nu(y) = \frac{\log\log{y}}{Z\log{y}}$$ for some constant $Z > 0$. Further, suppose that $B_i$ are positive constants and $f(Y)$ is a non-decreasing function such that such that if $\sigma_1 < \sigma < 1$, then 
\begin{equation}\label{eqn:zd_shape2}
    N(\sigma,Y) \leq B_1 (\log{Y})^{B_2} f(Y)^{1-\sigma} + B_3 (\log{Y})^{B_4} .
\end{equation}
If $x \geq e f(x^{\alpha})$ and $x \geq e^{1000}$, then
\begin{align*}
    \frac{\left| \theta(x+h) - \theta(x) - h \right|}{h} 
    &\leq \frac{\alpha x^{\sigma_1 + \alpha - 1} \log{x}}{\pi} + 2 B_3 (\alpha\log{x})^{B_4} ((\alpha\log{x})^{-\frac{1}{Z\alpha}} - x^{\sigma_1 - 1}) \\
    &\hspace{0.5cm} + \frac{2 B_1 (\alpha\log{x})^{B_2} \log{x}}{\log(xf(x^{\alpha})^{-1})} \left((\alpha\log{x})^{-\frac{\log(x f(x^{\alpha})^{-1})}{Z\alpha\log{x}}} - (x f(x^{\alpha})^{-1})^{\sigma_1 - 1}\right) \\
    &\hspace{0.5cm} + \frac{2.52}{k x^{\alpha - \frac{1}{k}}} \left( \left(1 + k x^{-\frac{1}{k}}\right) \left(\log{x} + k x^{-\frac{1}{k}}\right)^{\frac{8}{10}} + (\log{x})^{\frac{8}{10}} \right) \\
    &\hspace{0.5cm} + \left( \alpha_1 \left(1 + k x^{-\frac{1}{k}}\right)^{\frac{1}{2}} - 0.999\right) \frac{x^{\frac{1}{k} - \frac{1}{2}}}{k} \\
    &\hspace{0.5cm} + \Bigg( \alpha_2 \left(1 + k x^{-\frac{1}{k}}\right)^{\frac{1}{3}} - 0.885\Bigg) \frac{x^{\frac{1}{k} - \frac{2}{3}}}{k} .
\end{align*}
If $k \geq 86$, then the constants $1/10$, $2.52$, and $8/10$ can be replaced with $1/85$, $12.782$, and $1/10$ respectively.
\end{proposition}

\begin{remark}
In Proposition \ref{prop:theta_bounds}, we restrict our attention to $k \geq 11$ (resp.~$k\geq 86$) because we require $\alpha > 1/k$ and this condition cannot hold on $k \leq 10$ (resp.~$k\leq 85$). 
\end{remark}

Our proof of Proposition \ref{prop:theta_bounds} is deferred to Section \ref{ssec:theta_bounds}. Next, we use Proposition \ref{prop:theta_bounds} to prove Propositions \ref{prop:nighttime}, \ref{prop:CullyHugillLee_application}, and \ref{prop:nighttime2} which compute broad ranges of $x$ such that
\begin{equation}\label{eqn:suff_con}
    \theta(x + k x^{1 - \frac{1}{k}}) - \theta(x) > 0 
\end{equation}
holds for fixed choices of $k$. 
For example, the computations in Proposition \ref{prop:nighttime} demonstrate that if $k\geq 70$, then \eqref{eqn:suff_con} holds unconditionally for all $x \geq \exp(15\,951)$. 

\begin{proposition}\label{prop:nighttime}
Fix any $k \in [65,90]$ and suppose that $x_A$ and $x_B$ correspond to the values in Table \ref{tab:big}.
\begin{itemize}
    \item If $k \geq 86$ or $k \leq 66$, then \eqref{eqn:suff_con} holds for all $\log{x} \geq x_A$. 
    \item If $67 \leq k \leq 85$, then \eqref{eqn:suff_con} holds for all $\log{x} \geq x_B$.
\end{itemize}
\end{proposition}

\begin{proof}
Fix a $k \in [65,90]$ and suppose that $x_A$, $c_A$, $x_B$, $c_B$ correspond to the values in Table \ref{tab:big}. We only compute values for $x_B$, $c_B$ for every $k$ where there is a benefit to do so. Recall that we have $\zeta(\sigma + it) \neq 0$ in the region $\sigma \geq 1 - \widehat{Z}(t)^{-1} \log\log{t}/\log{t}$ and $t \geq H_R$, where $\widehat{Z}(t)$ has been defined in \eqref{eqn:Z_hat_def}. Further, we assert 
\begin{equation*}
\begin{split}
    \sigma_1 = 0.675, \quad 
    B_1 = 17.253, \quad 
    B_2 &= 3, \quad
    B_3 = 5.663, \\ 
    B_4 &= 2,\quad 
    f(t) = t^{\frac{8}{3}} (\log{t})^{2}, \quad\text{and}\quad \alpha = \frac{c_A}{k}
\end{split}
\end{equation*}
in Proposition \ref{prop:theta_bounds} to see \eqref{eqn:suff_con} holds for all $\log{x} \geq x_A$. Each of these parameters (except $c_A$) are chosen in light of \eqref{eqn:zd_K}. With this, the result is proved for each $k \geq 86$ and $k\leq 66$. For each $65 \leq k \leq 85$ and $x_A \geq \log{x} \geq x_B$, we assert 
\begin{equation*}
    \sigma_1 = 0.985, \quad 
    B_1 = C, \quad 
    B_2 = B_3 = B_4 = 0,\quad 
    f(t) = t^B, \quad\text{and}\quad \alpha = \frac{c_B}{k}
\end{equation*}
in Proposition \ref{prop:theta_bounds}, with $C$ and $B$ corresponding to the values in \eqref{eqn:Bellotti} and Table \ref{tab:Bellotti}. It follows that \eqref{eqn:suff_con} holds in the extended range $\log{x} \geq x_B$, and hence the result for each $67 \leq k \leq 85$ is complete.
\end{proof}

\begin{proposition}\label{prop:CullyHugillLee_application}
For any $k\geq 21$, the inequality \eqref{eqn:suff_con} is true for all $$4\cdot 10^{18} \leq x \leq e^{k\log\left(k\cdot 2.51949\cdot 10^{11} \left(1-\frac{1}{2.51949\cdot 10^{11}}\right)^{2}\right)}.$$
\end{proposition}

\begin{proof}
For each pair $(\Delta, x_0)$ in Table \ref{table1}, Cully-Hugill and the author \cite{CullyHugillLee, CullyHugillLeeCorrigendum} have proved that there exists at least one prime in the interval $(x(1 - \Delta^{-1}), x]$ for all $x\geq x_0$. It follows that if $x \geq x_0$, then there is at least one prime in the interval $(x(1 + kx^{-\frac{1}{k}})(1 - \Delta^{-1}), x(1 + kx^{-\frac{1}{k}})]$. Moreover,
\begin{align*}
    x(1 + kx^{-\frac{1}{k}})(1 - \Delta^{-1}) \geq x
    &\iff 1 + kx^{-\frac{1}{k}} \geq (1 - \Delta^{-1})^{-1} \\
    &\iff k x^{-\frac{1}{k}} \geq (1 - \Delta^{-1})^{-1} - 1 ,
\end{align*}
and
\begin{align*}
    (1 - \Delta^{-1})^{-1} - 1
    &= \int_{1-\Delta^{-1}}^{1} t^{-2}\,dt
    \leq \frac{\Delta^{-1}}{(1-\Delta^{-1})^2} .
\end{align*}
It follows that there is at least one prime in the interval $(x, x + kx^{1-\frac{1}{k}}]$ for all $x\geq x_0$ such that
\begin{align*}
    k x^{-\frac{1}{k}} \geq \frac{\Delta^{-1}}{(1-\Delta^{-1})^2}
    %&\iff x^{-\frac{1}{k}} \geq \frac{\Delta^{-1}}{k(1-\Delta^{-1})^2} \\
    &\iff \log{x} \leq k\log(k \Delta (1-\Delta^{-1})^2) .
\end{align*}
Since $\Delta \geq 3.90970\cdot 10^{7}$ for any $x\geq 4\cdot 10^{18}$, it follows that if $k\geq 21$, then there is at least one prime in the interval $(x, x + kx^{1-\frac{1}{k}}]$ for all  
\begin{equation*}
    4\cdot 10^{18} \leq x \leq e^{615.2373} < e^{21\log\left(21\cdot 3.90970\cdot 10^{7} \left(1-\frac{1}{3.90970\cdot 10^{7}}\right)^{2}\right)} .
\end{equation*}
Similarly, $\Delta \geq 2.51949\cdot 10^{11}$ for any $x\geq e^{600}$. Therefore, by extension there is at least one prime in the interval $(x, x + kx^{1-\frac{1}{k}}]$ for all  
\begin{equation*}
    4\cdot 10^{18} \leq x \leq e^{k\log\left(k\cdot 2.51949\cdot 10^{11} \left(1-\frac{1}{2.51949\cdot 10^{11}}\right)^{2}\right)} .
\end{equation*}
With this, we have proved the result.
\end{proof}

\begin{proposition}\label{prop:nighttime2}
Fix any $k \in [65,90]$. Suppose that $x_A$, $x_B$ correspond to the values in Table \ref{tab:big} and $\mathcal{Z}$, $x_{\mathcal{Z}}$, $c_{\mathcal{Z}}$ correspond to the values in Table \ref{tab:big2}. We also let $\widehat{x}(k) = x_A$ if $k \geq 86$ or $k\leq 66$ and $\widehat{x}(k) = x_B$ if $67 \leq k \leq 85$. If $\zeta(\sigma + it) \neq 0$ in the region
\begin{equation*}
    1 - \frac{\log\log{|t|}}{\mathcal{Z}\log{|t|}} \leq \sigma \leq 1 
    \quad\text{and}\quad
    |t| \leq e^{\widehat{x}(k) \left(\frac{c_{\mathcal{Z}}}{k}\right)} ,
\end{equation*}
then \eqref{eqn:suff_con} holds for all $\exp(x_{\mathcal{Z}}) \leq x \leq \exp(\widehat{x}(k))$.
\end{proposition}

\begin{proof}
For any fixed $k \in [65,90]$, assert
\begin{align*}
    Z = \mathcal{Z}, \quad
    \sigma_1 = 0.675, \quad
    B_1 = 17.253, \quad 
    B_2 &= 3, \quad
    B_3 = 5.663, \\ 
    B_4 &= 2, \quad 
    f(t) = t^{\frac{8}{3}} (\log{t})^{2}\quad\text{and}\quad 
    \alpha = \frac{c_{\mathcal{Z}}}{k}
\end{align*}
in Proposition \ref{prop:theta_bounds}. These parameters are chosen in light of our assumptions and \eqref{eqn:zd_K}, and these choices imply \eqref{eqn:suff_con} holds in the range $\log{x} \geq x_{\mathcal{Z}}$. Further, since we only need to consider $x\leq \exp(\widehat{x}(k))$, we only require the desired zero-free region to hold for all $t$ such that $|t| \leq x^{\alpha} \leq \exp(\widehat{x}(k) \alpha)$; this can be deduced by analysing the proof of Proposition \ref{prop:theta_bounds}. 
\end{proof}

\begin{remark}
To make optimal choices for $\alpha$ in Propositions \ref{prop:nighttime} and \ref{prop:nighttime2}, we tested $\alpha = (1 + \varepsilon)/k$ for systematically incremented choices of $\varepsilon > 0$ until the choice $\alpha = (1 + \varepsilon)/k$ yielded the best computation. 
In Proposition \ref{prop:nighttime2}, we have computed $\mathcal{Z}$ optimally by systematically decreasing choices of $\mathcal{Z}$ until the largest $\mathcal{Z}$ has been found such that the result is proved for all $x$ not covered by Propositions \ref{prop:nighttime}-\ref{prop:CullyHugillLee_application}. 
For simplicity, we made choices for each $\varepsilon$ and $\mathcal{Z}$ to five decimal places and three decimal places respectively. 
\end{remark}

\subsection{Proof of Proposition \ref{prop:theta_bounds}}\label{ssec:theta_bounds}

Make the same assumptions as in the statement of Proposition \ref{prop:theta_bounds}. 
To begin, recall that Cully-Hugill and Johnston \cite{CH_J_24} have proved that if $x \geq e^{1\,000}$ and $\max\{51,\log{x}\} < 2T < (x^{1/10} - 2)/4$, then there exists a point $T \leq T^* \leq 2T$ such that
\begin{equation}\label{eqn:ETPF}
    \left| \psi(x+h) - \psi(x) - h + \Sigma(T^*) \right| \leq \frac{1.26}{T} \left( (x+h) (\log(x+h))^{\frac{8}{10}} + x (\log{x})^{\frac{8}{10}} \right) ,
\end{equation}
where
\begin{equation*}
    \Sigma(T^*) 
    = \sum_{|\gamma | \leq T^*} \frac{(x+h)^{\varrho} - x^{\varrho}}{\varrho}
    = \sum_{|\gamma | \leq T^*} \int_{x}^{x+h} t^{\varrho - 1} \,dt .
\end{equation*}
The constants $1/10$, $1.26$, and $8/10$ can be replaced with $1/85$, $6.391$, and $1/10$ respectively by analysing the computations in \cite{CH_J_24}. We proceed with $1/10$, $1.26$, and $8/10$ for the remainder of this proof, but one can repeat the same arguments with these updated values to obtain the advertised refinement. Next, recall from \cite[Eqns.~(21),~(31)]{Costa} and \cite[Cor.~5.1]{Broadbent} that for all $\log{x}\geq 300$, we have
\begin{equation}\label{eqn:diffs}
    0.999 x^\frac{1}{2} + 0.885 x^\frac{1}{3}
    < \psi(x)-\theta(x)
    < \alpha_1 x^\frac{1}{2} + \alpha_2 x^\frac{1}{3} .
\end{equation}
Improvements are available for the constants $0.999$ and $0.885$ on the final range of $x$ under consideration in this proof, however these updates would make no difference to our final computations. 
It follows from \eqref{eqn:diffs} that if $\log{x}\geq 300$, then
\begin{align*}
    | \psi(x+h) - \psi(x) &- \theta(x+h) + \theta(x) | \\
    &< \Bigg( \alpha_1 \left(1 + \frac{h}{x}\right)^{\frac{1}{2}} - 0.999\Bigg) \sqrt{x} 
    + \Bigg( \alpha_2 \left(1 + \frac{h}{x}\right)^{\frac{1}{3}} - 0.885\Bigg) x^{\frac{1}{3}} \\
    &\leq \left( \alpha_1 \left(1 + k x^{-\frac{1}{k}}\right)^{\frac{1}{2}} - 0.999\right) \sqrt{x} 
    + \Bigg( \alpha_2 \left(1 + k x^{-\frac{1}{k}}\right)^{\frac{1}{3}} - 0.885\Bigg) x^{\frac{1}{3}} .
\end{align*}
Therefore, this and \eqref{eqn:ETPF} imply
\begin{equation}\label{eqn:ETPF_applied}
\begin{split}
    \left| \theta(x+h) - \theta(x) - h \right| &\leq |\Sigma(T^*)| + \frac{1.26}{T} \left( (x+h) (\log(x+h))^{\frac{8}{10}} + x (\log{x})^{\frac{8}{10}} \right) \\
    &\hspace{3cm} + \left( \alpha_1 \left(1 + k x^{-\frac{1}{k}}\right)^{\frac{1}{2}} - 0.999\right) \sqrt{x} \\
    &\hspace{3cm} + \Bigg( \alpha_2 \left(1 + k x^{-\frac{1}{k}}\right)^{\frac{1}{3}} - 0.885\Bigg) x^{\frac{1}{3}} .
\end{split}
\end{equation}
We bound $|\Sigma(T^*)|$ by observing
\begin{align*}
    |\Sigma(T^*)| 
    &\leq \sum_{|\gamma | \leq 2T} \int_{x}^{x+h} t^{\beta - 1} \,dt 
    \leq h \sum_{|\gamma | \leq 2T} x^{\beta - 1} .
\end{align*}
Further, we also have
\begin{align*}
    \sum_{|\gamma|\leq 2T} (x^{\beta - 1} - x^{-1})
    &= \sum_{|\gamma|\leq 2T} \int_0^{\beta} x^{\sigma - 1}\log{x}\,dy \\
    &= 2 \left(\int_0^{\sigma_1} + \int_{\sigma_1}^{1 - \nu(2T)}\right) N(\sigma,2T) x^{\sigma - 1} \log{x} \,d\sigma .
\end{align*}
Therefore, \eqref{eqn:a_bound_of_convenience} and \eqref{eqn:zd_shape2} imply
\begin{align*}
    \frac{|\Sigma(T^*)|}{2h}
    &\leq N(2T) \left(\frac{1}{x} + \int_0^{\sigma_1} x^{\sigma - 1} \log{x} \,d\sigma \right) + \int_{\sigma_1}^{1 - \nu(2T)} N(\sigma,2T) x^{\sigma - 1} \log{x} \,d\sigma \\
    &= x^{\sigma_1 - 1} N(2T) + \int_{\sigma_1}^{1 - \nu(2T)} N(\sigma,2T) x^{\sigma - 1} \log{x} \,d\sigma \\
    &\leq \frac{x^{\sigma_1 - 1} T\log(2T)}{\pi} 
    + B_1 (\log(2T))^{B_2} \log{x} \int_{\sigma_1}^{1 - \nu(2T)} (x f(2T)^{-1})^{\sigma - 1} \,d\sigma \\
    &\hspace{7cm} + B_3 (\log(2T))^{B_4} (x^{- \nu(2T)} - x^{\sigma_1 - 1}) .
\end{align*}
We assume $x \geq e f(2T)$ to ensure the integral evaluates to
\begin{equation*}
    \int_{\sigma_1}^{1 - \nu(2T)} (x f(2T)^{-1})^{\sigma - 1} \,d\sigma
    = \frac{(x f(2T)^{-1})^{- \nu(2T)} - (x f(2T)^{-1})^{\sigma_1 - 1}}{\log(x f(2T)^{-1})} .
\end{equation*}
Insert these observations into \eqref{eqn:ETPF_applied} to see
\begin{equation*}%\label{eqn:ETPF_applied_2}
\begin{split}
    \frac{\left| \theta(x+h) - \theta(x) - h \right|}{h} &\leq \frac{x^{\sigma_1 - 1} (2T) \log(2T)}{\pi} + 2 B_3 (\log(2T))^{B_4} (x^{- \nu(2T)} - x^{\sigma_1 - 1}) \\
    &\hspace{1cm} + \frac{2 B_1 (\log(2T))^{B_2} \log{x}}{\log(xf(2T)^{-1})} ((x f(2T)^{-1})^{-\nu(2T)} - (x f(2T)^{-1})^{\sigma_1 - 1}) \\
    &\hspace{1cm} + \frac{2\cdot 1.26 x}{h(2T)} \left( \left(1+\frac{h}{x}\right) (\log(x+h))^{\frac{8}{10}} + (\log{x})^{\frac{8}{10}} \right) \\
    &\hspace{1cm} + \left( \alpha_1 \left(1 + k x^{-\frac{1}{k}}\right)^{\frac{1}{2}} - 0.999\right) \frac{\sqrt{x}}{h} \\
    &\hspace{1cm} + \Bigg( \alpha_2 \left(1 + k x^{-\frac{1}{k}}\right)^{\frac{1}{3}} - 0.885\Bigg) \frac{x^{\frac{1}{3}}}{h} .
\end{split}
\end{equation*}
Finally, we assert $2T = x^{\alpha}$ such that $1/k < \alpha < 1/10$ and $\alpha < 1 - \sigma_1$ to ensure that the final bound decreases as $x\to\infty$. Proposition \ref{prop:theta_bounds} follows naturally, upon combination with the following simple observations:
\begin{equation*}
    x^{-\nu(x^{\alpha})} 
    %= x^{-\frac{\log(\alpha\log{x})}{Z\alpha\log{x}}}
    = (\alpha\log{x})^{-\frac{1}{Z\alpha}}
    \quad\text{and}\quad
    \log(x+h) \leq \log{x} + \frac{h}{x} .
\end{equation*}

\section{Primes Between Consecutive Powers}\label{sec:pbcp}

Let $a_n \geq 1$ be an increasing sequence of positive integers. For any fixed $k\geq 3$, there is at least one prime in the interval $(a_{n}^k,a_{n+1}^k)$ for every integer $n\geq n_0$ if and only if $\theta(a_{n+1}^k) - \theta(a_{n}^k) > 0$. If $n = 1$, then $a_2 \geq 2$ and it is known that there is at least one prime in the interval $(1^{k},2^{k})$, so we freely restrict our attention to $n \geq 2$. 

Note that
\begin{equation*}
    a_{n+1}^k - a_{n}^k 
    %= (a_{n+1} - a_{n}) \sum_{j=0}^{k-1} a_{n+1}^{k-1-j} a_{n}^{j}
    \geq (a_{n+1} - a_{n}) k a_{n}^{k-1}
    \geq k a_{n}^{k-1} .
\end{equation*}
Therefore, to prove the condition $\theta(a_{n+1}^k) - \theta(a_n^k) > 0$ holds for all $a_n \geq 2$, it suffices to prove $\theta(a_{n}^k + k a_{n}^{k-1}) - \theta(a_{n}^k) > 0$. Using the substitution $x = a_n^{k}$ this is equivalent to the sufficient condition \eqref{eqn:suff_con}. 

We also have $a_n^k \geq 2^{k} > 3.68\cdot 10^{19}$ on $n \geq 2$ for every $k \geq 65$. Therefore, for any fixed $k \geq 65$, to prove there is at least one prime in the interval $(a_{n}^k,a_{n+1}^k)$ for every integer $n\geq 1$, it suffices to prove \eqref{eqn:suff_con} for all $x \geq 3.68\cdot 10^{19}$.

\begin{proof}[Proof of Theorem \ref{thm:CHJ_extended}]
Fix $k = 86$ and assert $a_n = n$. Proposition \ref{prop:nighttime} tells us the inequality \eqref{eqn:suff_con} is true for all $x \geq \exp(2\,010)$ and Proposition \ref{prop:CullyHugillLee_application} tells us \eqref{eqn:suff_con} is true for all $3.68\cdot 10^{19} \leq x \leq \exp(2\,640)$. Therefore, \eqref{eqn:suff_con} is true for all $x\geq 3.68\cdot 10^{19}$ and the result follows.
\end{proof}

\begin{proof}[Proof of Theorem \ref{thm:CHJ_extended_bigger_n}]
Fix $k = 70$, $N \geq 0$, and assert $a_n = n(1 + \lfloor N/n \rfloor)$ for each integer $n\geq 1$. Proposition \ref{prop:nighttime} implies the inequality \eqref{eqn:suff_con} is true for all $x \geq \exp(15\,951)$ and Proposition \ref{prop:CullyHugillLee_application} tells us \eqref{eqn:suff_con} is true for all $3.68\cdot 10^{19} \leq x \leq \exp(2\,135)$. All that remains is to prove the inequality \eqref{eqn:suff_con} is true for all
\begin{equation}\label{eqn:quick}
    \exp(2\,135) < x < \exp(15\,951)
    \iff \exp\left(\frac{2\,135}{k}\right) < a_n < \exp\left(\frac{15\,951}{k}\right) .
\end{equation}
Now, the definition of $a_n$ implies $n \leq a_n \leq n + N$. It follows that \eqref{eqn:quick} is true when
\begin{align*}
    \exp\left(\frac{2\,135}{k}\right) < n 
    \quad\text{and}\quad
    N < \exp\left(\frac{15\,951}{k}\right) - \exp\left(\frac{2\,135}{k}\right) .
\end{align*}
Therefore, if $N \geq \exp\left(\frac{15\,951}{k}\right) - \exp\left(\frac{2\,135}{k}\right)$, then there are no $n\geq 1$ such that $a_n$ lies in the interval given in \eqref{eqn:quick}. It follows that if $N \geq \exp\left(\frac{15\,951}{k}\right) - \exp\left(\frac{2\,135}{k}\right)$, then the result is proved for every $n\geq 1$.
\end{proof}

\begin{proof}[Proof of Theorem \ref{thm:night}]
Fix $k \in \{85, 80, 75, 70\}$ and assert $a_n = n$. Proposition \ref{prop:nighttime} tells us \eqref{eqn:suff_con} is true for all $x \geq \exp(\widehat{x}(k))$, where $\widehat{x}(k)$ is defined as in Proposition \ref{prop:nighttime2}. Further, Proposition \ref{prop:nighttime2} certifies that if $\zeta(\sigma + it) \neq 0$ in the region
\begin{equation*}
    1 - \frac{\log\log{|t|}}{z_k\log{|t|}} \leq \sigma \leq 1 
    \quad\text{and}\quad
    |t| \leq e^{\widehat{x}(k) \left(\frac{c_{z_k}}{k}\right)} ,
\end{equation*}
then \eqref{eqn:suff_con} holds in the extended range $x \geq \exp(x_{\mathcal{Z}})$, where values for $x_{\mathcal{Z}}$ are given in Table \ref{tab:big2}. Since
\begin{equation*}
    x_{\mathcal{Z}} \leq k\log\left(k\cdot 2.51949\cdot 10^{11} \left(1-\frac{1}{2.51949\cdot 10^{11}}\right)^{2}\right) ,
\end{equation*}
it would follow that \eqref{eqn:suff_con} is true for all $x \geq 3.68\cdot 10^{19}$, which implies the result. Note that the values of $z_k$ and $T_k$ presented in Table \ref{tab:constants} correspond to the values of $\mathcal{Z}$ and $\widehat{x}(k) c_{\mathcal{Z}}/k$ associated to $k$ in Table \ref{tab:big2}.
\end{proof}

\appendix

\section{Tables}\label{app:tables}

We store the tables referred to throughout this paper in this appendix.

\begin{table}[h!]
\centering
\begin{minipage}{0.48\linewidth}
\centering
\begin{tabular}{| c c c c c |}
\hline
$\log{T_0}$ & $\log{T_1}$ & $\alpha_0$ & $C$ & $B$ \\
\hline
80    & 90      & 0.985 & 370655.73        & 5.216 \\
90    & 100     & 0.985 & 425721.47        & 4.831 \\
100   & 110     & 0.985 & 488901.14        & 4.513 \\
110   & 120     & 0.985 & 545744.21        & 4.264 \\
120   & 130     & 0.985 & 629490.27        & 4.032 \\
130   & 140     & 0.985 & 694045.43        & 3.855 \\
140   & 150     & 0.985 & 771373.78        & 3.696 \\
150   & 160     & 0.985 & 825913.07        & 3.572 \\
160   & 170.2   & 0.985 & 909966.07        & 3.448 \\
\hline
\end{tabular}
\end{minipage}
\hfill
\begin{minipage}{0.48\linewidth}
\centering
\begin{tabular}{| c c c c c |}
\hline
$\log{T_0}$ & $\log{T_1}$ & $\alpha_0$ & $C$ & $B$ \\
\hline
170.2 & 500     & 0.985 & $1.12\cdot 10^{6}$ & 3.337 \\
500   & 1000    & 0.985 & $6.23\cdot 10^{6}$ & 2.152 \\
1000  & 1500    & 0.985 & $2.12\cdot 10^{7}$ & 1.820 \\
1500  & 2000    & 0.985 & $6.47\cdot 10^{7}$ & 1.684 \\
2000  & 2500    & 0.985 & $1.73\cdot 10^{8}$ & 1.610 \\
2500  & 3000    & 0.985 & $2.56\cdot 10^{8}$ & 1.577 \\
3000  & 481958  & 0.985 & $5.76\cdot 10^{8}$ & 1.551 \\
481958  & $6.7\cdot 10^{12}$  & 0.985 & $1.62\cdot 10^{11}$ & 1.448 \\
\hline
\end{tabular}
\end{minipage}
\caption{Admissible computations for $\alpha_0$, $B$, and $C$ in the bound \eqref{eqn:Bellotti} in the intervals $[T_0, T_1]$.}
\label{tab:Bellotti}
\end{table}

\begin{table}[]
\centering

\begin{minipage}{0.48\linewidth}
\centering
\begin{tabular}{|c||cc||cc|}
\hline
$k$ & $x_A$ & $c_A$ & $x_B$ & $c_B$ \\
\hline
$90$ & $1997$ & $1.29478$ & - & - \\
$89$ & $1999$ & $1.27912$ & - & - \\
$88$ & $2001$ & $1.26348$ & - & - \\
$87$ & $2005$ & $1.24663$ & - & - \\
$86$ & $2010$ & $1.22924$ & - & - \\
$85$ & $6162$ & $1.08141$ & $5587$ & $1.21712$ \\
$84$ & $6828$ & $1.07501$ & $5587$ & $1.20280$ \\
$83$ & $7546$ & $1.07001$ & $5587$ & $1.18848$ \\
$82$ & $8450$ & $1.06231$ & $5587$ & $1.17416$ \\
$81$ & $9685$ & $1.06001$ & $5588$ & $1.15963$ \\
$80$ & $10986$ & $1.05001$ & $5588$ & $1.14532$ \\
$79$ & $12799$ & $1.04401$ & $5588$ & $1.13100$ \\
$78$ & $15286$ & $1.04001$ & $5589$ & $1.11648$ \\
\hline
\end{tabular}
\end{minipage}
\hfill
\begin{minipage}{0.48\linewidth}
\centering
\begin{tabular}{|c||cc||cc|}
\hline
$k$ & $x_A$ & $c_A$ & $x_B$ & $c_B$ \\
\hline
$77$ & $18427$ & $1.03221$ & $5592$ & $1.10158$ \\
$76$ & $23736$ & $1.03001$ & $5601$ & $1.08553$ \\
$75$ & $29735$ & $1.02151$ & $5880$ & $1.07041$ \\
$74$ & $43186$ & $1.02001$ & $6561$ & $1.06001$ \\
$73$ & $58051$ & $1.01212$ & $7554$ & $1.05201$ \\
$72$ & $96051$ & $1.01001$ & $9036$ & $1.04341$ \\
$71$ & $212515$ & $1.01001$ & $11620$ & $1.03505$ \\
$70$ & $580009$ & $1.01001$ & $15951$ & $1.03001$ \\
$69$ & $2133517$ & $1.01001$ & $23148$ & $1.02001$ \\
$68$ & $12318302$ & $1.01001$ & $43001$ & $1.01051$ \\
$67$ & $37420887$ & $1.01001$ & $238050$ & $1.00301$ \\
$66$ & $39361821$ & $1.01001$ & - & - \\
$65$ & $41432145$ & $1.01001$ & - & - \\
\hline
\end{tabular}
\end{minipage}
    \caption{Admissible $x_A$, $c_A$, $x_B$, $c_B$ for $k\in [65,90]$ in Proposition \ref{prop:nighttime}.}
    \label{tab:big}
\end{table}

\begin{table}[]
\centering

\begin{minipage}{0.48\linewidth}
\centering
\begin{tabular}{| c | c |}
    \hline 
    $\log x_0$ & $\Delta$ \\ 
    \hline
    $\log(4\cdot 10^{18})$ & $3.90970\cdot 10^{7}$ \\
    43  & $4.18168 \cdot10^{7}$ \\
    46  & $1.63940\cdot 10^{8}$ \\
    50  & $1.06120\cdot 10^{9}$ \\
    55  & $1.02884\cdot 10^{10}$ \\
    60  & $7.69184\cdot 10^{10}$ \\
    75  & $1.74043\cdot 10^{11}$ \\
    \hline
\end{tabular}
\end{minipage}
\hfill
\begin{minipage}{0.48\linewidth}
\centering
\begin{tabular}{| c | c |}
    \hline 
    $\log x_0$ & $\Delta$ \\ 
    \hline
    90  & $1.84304\cdot 10^{11}$ \\
    105 & $1.91886\cdot 10^{11}$ \\
    120 & $1.97917\cdot 10^{11}$ \\
    135 & $2.02553\cdot 10^{11}$ \\
    150 & $2.07053\cdot 10^{11}$ \\
    300 & $2.30126\cdot 10^{11}$ \\
    600 & $2.51949\cdot 10^{11}$ \\
    \hline
\end{tabular}
\end{minipage}
\caption{Admissible pairs of $x_0$ and $\Delta$.}
\label{table1}
\end{table}

\begin{table}[]
\centering

\begin{minipage}{0.45\linewidth}
\centering
\begin{tabular}{|c|cccc|}
\hline
$k$ & $\mathcal{Z}$ & $x_{\mathcal{Z}}$ & $c_{\mathcal{Z}}$ & $\frac{\widehat{x}(k) c_{\mathcal{Z}}}{k}$ \\
\hline
$90$ & $20.525$ & $2767$ & $1.02401$ & - \\
$89$ & $20.281$ & $2734$ & $1.02411$ & - \\
$88$ & $20.039$ & $2704$ & $1.02442$ & - \\
$87$ & $19.795$ & $2672$ & $1.02352$ & - \\
$86$ & $19.460$ & $2640$ & $1.01176$ & - \\
$85$ & $17.270$ & $2609$ & $1.15401$ & $75.85240$ \\
$84$ & $17.055$ & $2577$ & $1.15402$ & $76.75607$ \\
$83$ & $16.841$ & $2545$ & $1.15413$ & $77.68824$ \\
$82$ & $16.627$ & $2514$ & $1.15423$ & $78.64248$ \\
$81$ & $16.413$ & $2482$ & $1.15434$ & $79.63521$ \\
$80$ & $16.198$ & $2450$ & $1.15451$ & $80.64252$ \\
$79$ & $15.985$ & $2419$ & $1.15501$ & $81.69868$ \\
$78$ & $15.770$ & $2387$ & $1.15501$ & $82.76091$ \\
\hline
\end{tabular}
\end{minipage}
\hfill
\begin{minipage}{0.45\linewidth}
\centering
\begin{tabular}{|c|cccc|}
\hline
$k$ & $\mathcal{Z}$ & $x_{\mathcal{Z}}$ & $c_{\mathcal{Z}}$ & $\frac{\widehat{x}(k) c_{\mathcal{Z}}}{k}$ \\
\hline
$77$ & $15.549$ & $2355$ & $1.16001$ & $84.24384$ \\
$76$ & $15.336$ & $2324$ & $1.16001$ & $85.48968$ \\
$75$ & $15.122$ & $2292$ & $1.16001$ & $90.94478$ \\
$74$ & $14.909$ & $2261$ & $1.16001$ & $102.84899$ \\
$73$ & $14.695$ & $2229$ & $1.16001$ & $120.0372$ \\
$72$ & $14.482$ & $2198$ & $1.16001$ & $145.58125$ \\
$71$ & $14.268$ & $2166$ & $1.16001$ & $189.84952$ \\
$70$ & $14.055$ & $2135$ & $1.16001$ & $264.33314$ \\
$69$ & $13.841$ & $2103$ & $1.16001$ & $389.15814$ \\
$68$ & $13.628$ & $2072$ & $1.16001$ & $733.55279$ \\
$67$ & $13.414$ & $2040$ & $1.16001$ & $4121.49822$ \\
$66$ & $13.201$ & $2009$ & $1.16001$ & $691819.78755$ \\
$65$ & $12.987$ & $1977$ & $1.16001$ & $739410.80802$ \\
\hline
\end{tabular}
\end{minipage}
    \caption{Admissible $\mathcal{Z}$, $x_{\mathcal{Z}}$, $c_{\mathcal{Z}}$ for $k\in [65,90]$ in Proposition \ref{prop:nighttime2}.}
    \label{tab:big2}
\end{table}

\bibliographystyle{amsplain} 
\bibliography{references}

\end{document}